\documentclass[11pt]{article} \usepackage{ifpdf}

\usepackage[british]{babel} \newcommand{\Date}{\today}

\usepackage{a4wide} 
\usepackage[final]{showkeys}
\usepackage[normalem]{ulem}

\usepackage{calc, xspace} \usepackage{amsmath, amsthm, amssymb,
  amsfonts, euscript}

\usepackage{natbib} \setcitestyle{round}

\ifpdf \usepackage{xcolor} \input dvipsnam.def \else
\usepackage{color} \fi \usepackage{graphicx}

\definecolor{linkcolor}{named}{Maroon}
\definecolor{citecolor}{named}{PineGreen}
\definecolor{urlcolor}{named}{RoyalPurple}
\definecolor{okcolor}{named}{OliveGreen}
\definecolor{alertcolor}{named}{BrickRed}

\ifpdf \DeclareGraphicsExtensions{.png,.jpg}%
\else \DeclareGraphicsExtensions{.eps,.ps,.png,.jpg}%
\fi

\theoremstyle{plain} \newtheorem{thm}{Theorem}
\newtheorem{lem}[thm]{Lemma} 
\theoremstyle{definition} \newtheorem{defn}[thm]{Definition}
\theoremstyle{remark}

\newcommand{\Ex}[1]{\operatorname{\mathbb{E}}\left[#1\right]}

\newcommand{\prob}[1]{\operatorname{\mathbb{P}}\left(#1\right)}

\newcommand{\Var}[1]{\operatorname{Var}\left(#1\right)}
\newcommand{\bra}[1]{\left(#1\right)}

\newcommand{\NN}{\mathbb N} \newcommand{\ZZ}{\mathbb Z}
 \newcommand{\set}[1]{\left\{#1\right\}}
\newcommand{\squ}[1]{\left[#1\right]} 
\newcommand{\abs}[1]{\vert#1\vert}
\newcommand{\norm}[1]{\Vert#1\Vert}

\newcommand{\indev}[1]{\mathbf{1}_{\squ{#1}}}

\newcommand{\eps}{\varepsilon} \newcommand{\degree}{\mathrm{deg}}
\newcommand{\tr}{\mathrm{tr}} \newcommand{\sgn}[1]{\textrm{sgn}(#1)}
\renewcommand{\Re}{\mathrm{Re}}

\newcommand{\mebfussy}[1]{#1'}
\newcommand{\tv}[1]{\norm{#1}}
\newcommand{\tmix}[1]{t^{#1}_{\text{mix}}}
\usepackage[ \ifpdf pdftex, pdfstartview=FitH, \fi ]{hyperref}
\hypersetup{pdfstartview=FitH} \hypersetup{citecolor=citecolor}
\hypersetup{linkcolor=linkcolor}
\hypersetup{urlcolor=urlcolor} \hypersetup{colorlinks=true}

\begin{document}

 \title{Mixing Time and Cutoff for a Random Walk on the \\Ring of Integers mod
   $n$}
 \author{%
   \href{http://maths.york.ac.uk/www/meb505}{Michael E. Bate} \quad
   and\quad \href{http://maths.york.ac.uk/www/sbc502}{Stephen
     B. Connor}} \date{\Date}

 \maketitle

 \begin{abstract}\noindent
   We analyse a random walk on the ring of integers mod $n$, which at
   each time point can make an additive `step' or a multiplicative
   `jump'. When the probability of making a jump tends to zero as an
   appropriate power of $n$ we prove the existence of a total
   variation pre-cutoff for this walk. In addition, we show that
   the process obtained by subsampling our walk at jump times exhibits
   a true cutoff, with mixing time dependent on
   whether the step distribution has zero mean.

   \medskip
   \noindent{\bf Keywords:} random walk; mixing time; cutoff
   phenomenon; pre-cutoff; group representation theory; integers mod $n$; random
   number generation

   \medskip
   \noindent 2010 Mathematics Subject Classification: \\
   Primary 60J10
 \end{abstract}


 \section{Introduction}\label{sec:introduction}

 In this note we consider a random walk $X=X^{(n)}$ on $\ZZ_n =
 \ZZ/n\ZZ$ (where $n$ is \emph{odd}) defined as follows: $X_0=0$, and
 for $t\geq 1$,
 \begin{equation}\label{eqn:process_defn}
   X_t = \begin{cases}
     X_{t-1}+\mebfussy\xi_t \mod n & \text{with probability $1-p_n$} \\
     2X_{t-1} \mod n & \text{with probability $p_n$,}
   \end{cases}
 \end{equation}
 where $\{\mebfussy\xi_t\}$ are a set of i.i.d. random variables with
 finite support $B\subset\ZZ$, whose distribution does not vary with
 $n$. We denote the mean and variance of $\mebfussy\xi$ by $\mu$ and
 $\sigma^2$ respectively. We will refer to an `addition' move as a
 `step', and to a `multiplication' move as a `jump'. To ensure that
 $X$ is irreducible we assume that the group $\langle B_n,+ \rangle$
 is not a proper subgroup of $\ZZ_n$ for any odd $n$, where $B_n =
 \set{z \mod n:z\in B}$. Furthermore, since $n$ is odd, multiplication
 by 2 is an invertible operation, and thus $X$ is ergodic with uniform
 equilibrium distribution $\pi_n$ on $\ZZ_n$.

 Define the total variation distance from $\pi_n$ of a probability
 distribution $P$ on $\ZZ_n$ by
 \[ \tv{P-\pi_n} = \max_{A\subset \ZZ_n}\abs{P(A) - \pi_n(A)} =
 \frac{1}{2} \sum_{s\in\ZZ_n}\abs{P(s) - 1/n} \,. \] The
 $\eps$\emph{-mixing time} of $X$ is defined for any $\eps\in[0,1]$ as
 \[ \tmix{}(\eps) = \min\{t\,:\,\tv{\mathbb
   P(X_t\in\cdot)-\pi_n(\cdot)}\leq \eps\}\,, \]
 and the value of
 $\tmix{}(1/4)$ is commonly referred to as the `mixing time' of $X$.

 A number of authors have previously considered random processes of
 the form
 \[ X_t = a_t X_{t-1} + b_t \mod n; \] these processes are similar to
 schemes used for random number generation, a link which has naturally
 motivated interest in bounding their mixing times. A nice
 introduction to the area can be found in \citet[Chapter
 6]{Terras1999}. The earliest such work appears to be that of
 \citet{Chung1987}, in which $a_t=a=2$ and $b_t$ is chosen uniformly
 from $\{-1,0,1\}$: they show that $O(\log n \log\log n)$ steps
 suffice for this walk to mix, and that $O(\log n\log\log n)$ steps
 are also necessary for $n$ of the form $2^m-1$; on the other hand,
 for almost all odd $n$, $1.02\log_2 n$ steps suffice. This
 (deterministic) act of doubling each time causes the process to mix
 significantly faster than when $a_t=1$ for all $t$ where, if $b_t$ is
 uniform on a finite set (and assuming that the resulting process is
 irreducible), the mixing time is of order $n^2$
 \citep{Diaconis1988,Saloff-Coste2004}.

 Rather more general results have been established in a series of
 works by Hildebrand. It is shown in his thesis \citep[Chapter
 3]{Hildebrand1992} that if $a_t=a$ for all $t$, and for fairly
 general choices of $b_t$ (which don't depend on $n$), $O(\log n
 \log\log n)$ steps suffice, and in fact for almost all $n$, $O(\log
 n)$ steps suffice.
 When $a_t$ is allowed to vary with $t$, a
 general upper bound for the mixing time is proved in
 \citet{Hildebrand1993}: using a recursive relation involving discrete
 Fourier transforms (of which more below), he shows that (unless
 $a_t=1$ always, $b_t=0$ always, or $a_t$ and $b_t$ can each take on
 only one value) $O((\log n)^2)$ time steps are always
 sufficient. Other related results can be found in
 \citet{Hildebrand1994, Hildebrand1994b}.

 A particularly interesting feature of these processes is the
 quantitatively different behaviour that can be obtained by making
 small changes to the distribution of $a_t$ and $b_t$. For example,
 \citet{Chung1987} remark upon the following curiosity to be found
 when $a_t=2$ and $b_t$ is supported on $\{-1,0,1\}$ with
 $\prob{b_t=1} = \prob{b_t=-1} = q$: if $q=1/4$ or $q=1/2$ then
 $O(\log n)$ steps suffice to make the total variation distance small;
 however, if $q=1/3$ then $O(\log n\log\log n)$ steps may be
 required. Similarly, \citet[Chapter 5]{Hildebrand1992} considers the
 situation where $b_t$ is uniform on $\pm 1$ and $a_t$ is supported on
 $\{2,(n+1)/2\}$, with $\prob{a_t=2} = p\in(0,1)$: the mixing time is
 shown to be at most $O((\log n)^m)$, where $m$ is 2 if $p=1/2$, and 1
 otherwise. If the distribution of $b_t$ is altered to uniform on
 $\{-1,0,1\}$ then $O((\log n\log\log n)^m)$ steps suffice.

 The principal difference between these earlier works and the process
 defined in \eqref{eqn:process_defn} is that here we allow the
 probability of a `jump', $p_n$, to depend on $n$. In particular, we
 are able to show that if $p_n$ tends to zero as a power of $n$, then
 our process exhibits a total variation \emph{pre-cutoff}.

 \begin{defn}
   \label{def:pre-cutoff}
   A sequence of chains $\{X^{(n)}\}_{n\in \NN}$, with $\eps$-mixing
   times $\{\tmix{(n)}(\eps)\}$, is said to exhibit a
   \emph{pre-cutoff} if it satisfies
   \[
   \sup_{0<\eps<1/2} \limsup_{n\to\infty}
   \frac{\tmix{(n)}(\eps)}{\tmix{(n)}(1-\eps)} < \infty \,.
   \]
 \end{defn}
 A stronger condition than pre-cutoff is that the sequence of chains
 exhibits a total variation \emph{cutoff}.

 \begin{defn}
   \label{def:cutoff}
   A sequence of Markov chains $\{X^{(n)}\}_{n\in \NN}$ is said to
   exhibit a \emph{total variation cutoff} at time $T_n$ with
   \emph{window size} $w_n$ if $w_n = o(T_n)$ and
   \begin{align*}
     \lim_{c\to\infty} \liminf_{n\to\infty}\tv{\mathbb
       P(X^{(n)}_{T_n-cw_n}\in\cdot) - \pi_n(\cdot)} &= 1\\
     \lim_{c\to\infty} \limsup_{n\to\infty}\tv{\mathbb
       P(X^{(n)}_{T_n+cw_n}\in\cdot) - \pi_n(\cdot)} &= 0 \,.
   \end{align*}
 \end{defn}
 Intuitively this says that as $n$ gets large the convergence to
 equilibrium, measured using total variation distance, happens in a
 negligible window of order $w_n$ around the cutoff time $T_n$. We
 remark that it is possible for the `right' and `left' window sizes in
 the above definition to be of different orders -- see
 \mbox{\citet{Connor-2010}} for an example. There has been much
 interest in studying the mixing times of Markov chains and proving
 the existence of cutoff phenomena: see \mbox{\citet{Levin2009}} and
 \mbox{\citet{Diaconis2011}} for recent introductions to the area, or
 \mbox{\citet{Saloff-Coste2004}} for a more analytical overview. In
 addition a number of natural sequences of Markov chains are known to
 exhibit pre-cutoff, with the question of whether they in fact exhibit
 a cutoff still open; these include card shuffles using
 cyclic-to-random transpositions \citep{Mossel2004} or
 random-to-random insertions \citep{Subag2013}, and a Gibbs sampler on
 the $n$-simplex \citep{Smith2014}.

 Throughout the remainder of this paper we shall simply write $X$ for
 $X^{(n)}$, with the understanding that we are in reality considering
 a sequence of processes on state spaces $\ZZ_n$ of increasing
 size. The main obstruction to analysing $X$ using standard techniques
 for random walks on groups is that the distribution of $X_k$ is not
 given by convolution of $k$ independent increment distributions. This
 problem can be overcome by (initially) restricting attention to the
 process $Y$ which is produced by subsampling $X$ at jump
 times. Denote the jump times of $X$ by $\tau_1,\tau_2,\dots$, and let
 $\tau_0=0$; then $Y_k := X_{\tau_k}$, with $Y_0 = X_0=0$. This
 process clearly satisfies $Y_k = \mebfussy Y_k \mod n$, where
 \begin{equation}\label{eqn:Y_k_and_tilde_S}
   \mebfussy Y_k =  \sum_{i=1}^k 2^{k+1-i}\mebfussy S_i  \quad\text{and}\quad
   \mebfussy S_i = \sum_{t=\tau_{i-1}+1}^{\tau_i-1}\mebfussy\xi_t\,.
 \end{equation}
 Here (and throughout) we use the convention that random variables
 with a prime take values in $\ZZ$, while those without take values in
 $\ZZ_n$. Thus $S_i=\mebfussy S_i \mod n$ is the change in $X$ due to
 steps taken between jump times $\tau_{i-1}$ and $\tau_i$. Like $X$,
 $Y$ is ergodic with uniform equilibrium distribution. From
 \eqref{eqn:Y_k_and_tilde_S}
 it is clear that the distribution of $Y_k$ is given by convolution of
 the distributions corresponding to the independent increments
 $\{2^{k+1-i}S_i\}$, and this will prove essential to our method for
 establishing an upper bound on the mixing time of both $X$ and $Y$ in
 Section~\ref{sec:upper bound}.

 In order to state our main results, we first need to establish a
 little more notation: we shall write $\sigma^2_{\mebfussy S}=
 \Var{\mebfussy S_i}$ and $T_n = \log_2(n/\sigma_{S'})$. Note that the
 length of time between jumps of $X$ has a $\text{Geometric}(p_n)$
 distribution,
 and a straightforward application of the conditional variance formula
 shows that
 \begin{equation}\label{eqn:var_S}
   \sigma^2_{\mebfussy S}=
   \frac{(1-p_n)(\mu^2+p_n\sigma^2)}{p_n^2} \,.
 \end{equation}
 Thus if $p_n\to0$,
 \begin{equation}\label{eqn:T_n-behaviour}
   T_n\sim\begin{cases}
     \log_2\left(\frac{n\sqrt{p_n}}{\sigma}\right) & \text{if $\mu=0$} \\
     \log_2\left(\frac{n{p_n}}{|\mu|}\right) & \text{otherwise.}
   \end{cases}
 \end{equation}

 Our main results are as follows.

\begin{thm}\label{thm:Y_cutoff}
  Suppose that $p_n = 1/(2n^\alpha)$ for some $\alpha\in(0,2)$ such
  that $T_n\to\infty$ as $n\to\infty$. Then $Y$ exhibits a total
  variation cutoff at time $T_n$, with cutoff window of size
  $O(1)$. Indeed, for sufficiently large $c>0$,
  \begin{enumerate}
  \item $\liminf_{n\to\infty}\tv{\prob{Y_{T_n-c}\in\cdot} -
      \pi_n(\cdot)} \geq 1-4^{1-c/3}$;
  \item $\limsup_{n\to\infty}\tv {\prob{Y_{T_n+c}\in\cdot} -
      \pi_n(\cdot)} \leq O(2^{-c})$.
  \end{enumerate}
\end{thm}

Note that the mixing time of $Y$ is relatively
insensitive to the distribution of the step lengths: as can be seen from \eqref{eqn:T_n-behaviour}, the
cutoff time $T_n$ essentially depends on $\mebfussy\xi$ only through
its mean, $\mu$; in the case of zero drift the mixing time is
asymptotically $(1-\alpha/2)\log_2 n$ ($0<\alpha<2$), while if
$\mu\neq 0$ the mixing is slightly faster, with cutoff at $(1-\alpha)
\log_2 n$ ($0<\alpha<1$).

\begin{thm}\label{thm:main_result} 
  Suppose that $p_n$ satisfies the same condition as in
  Theorem~\ref{thm:Y_cutoff}. Let $T_n^L = (2\ln 2) n^\alpha T_n$,
  $T_n^R = 2n^\alpha T_n$, $w_n^L = 2n^\alpha$ and $w_n^R =
  2n^\alpha\sqrt{T_n}$.  Then for sufficiently large $c>0$,
  \begin{enumerate}
  \item
    $\liminf_{n\to\infty}\tv{\mathbb{P}(X_{T_n^L-cw_n^L}\in\cdot)-\pi_n(\cdot)}\geq
    1-ae^{-c/2}$, for some finite constant $a$;
  \item
    $\limsup_{n\to\infty}\tv{\mathbb{P}(X_{T_n^R+cw_n^R}\in\cdot)-\pi_n(\cdot)}\leq
    2/c^2$.
  \end{enumerate}
\end{thm}
In particular, since $T_n^L/T_n^R = O(1)$, 
$w_n^L = o(T_n^L)$ and $w_n^R = o(T_n^R)$,
Theorem~\ref{thm:main_result} shows that $X$ exhibits a pre-cutoff,
with mixing time $\tmix{(n)}(\eps) = \Omega(n^\alpha \ln n)$.

\medskip
Throughout the rest of the paper we shall
work under the assumptions of Theorem~\ref{thm:Y_cutoff}.
In Section~\ref{sec:new_lower_bound} we establish lower
bounds on the mixing time of both $X$ and $Y$ (proving part 1 of
Theorems~\ref{thm:Y_cutoff} and \ref{thm:main_result}). In
Section~\ref{sec:upper bound} we prove the corresponding upper
bounds. Section~\ref{sec:conclusions} contains some concluding
remarks and open questions. 


\section{Lower bounds}\label{sec:new_lower_bound}
As is typical for many problems of this sort, finding lower bounds for
the mixing times of our two processes is significantly easier than
establishing upper bounds. The general approach in each case is to
find a suitably large subset of the state space which our chain has
negligible chance of hitting before the time of interest. 

\subsection{Lower bound for $Y$}\label{ssec:lower_for_Y}

An elementary calculation using
\eqref{eqn:Y_k_and_tilde_S} shows that $\Ex{Y_{T_n-c}'} \sim \sgn\mu
2^{1-c}n$, where we define $\sgn 0 =0$. Furthermore, $\Var{\mebfussy
Y_k} = 4(4^k-1)\sigma^2_{\mebfussy S}/3$, and so $\Var{\mebfussy Y_{T_n-c}} \le
4^{1-c}n^2/3$ for sufficiently large $n$. Now consider the interval
\[ A_n(c)= \set{z\in\ZZ_n\,:\, \abs{z-\Ex{ \mebfussy Y_{T_n-c}}}>d_cn}
, \] for some value $d_c\in(0,1/2)$ which we shall choose later, and
where $\abs\cdot$ represents the usual distance between
two numbers mod $n$. Note
that $\pi_n(A_n(c)) \ge 1 - 2d_c - 1/n$, and that (subject to this condition)
this set has been chosen to be as far away as possible from
$\Ex{\mebfussy Y_{T_n-c}}$.

We now use the fact that $Y$ is equal to $\mebfussy Y \mod n$, along
with Chebychev's inequality, to bound the probability that $Y_{T_n-c}$
belongs to our chosen set $A_n(c)$:
\[
\prob{Y_{T_n-c} \in A_n(c)} \leq \prob{\abs{\mebfussy Y_{T_n-c} -
    \Ex{\mebfussy Y_{T_n-c} }}> d_cn} \leq \frac{\Var{\mebfussy
    Y_{T_n-c}}}{(d_cn)^2} \leq\frac{4^{1-c}}{3d_c^2} \,.
\]

Thus the set $A_n(c)$ satisfies
\[ \pi_n(A_n(c)) - \prob{Y_{T_n-c}\in A_n(c)} \geq 1-2d_c -1/n -
\frac{4^{1-c}}{3d_c^2} \,. \] Finally, taking $d_c =
\bra{4^{1-c}/3}^{1/3}$ yields the claimed left hand window of the
cutoff in part 1 of Theorem~\ref{thm:Y_cutoff}:
\[
\liminf_{n\to\infty}\tv{\prob{Y_{T_n-c}\in\cdot}-\pi_n(\cdot)}\geq 1
-\left(\frac{9}{4^{c-1}}\right)^{1/3} \geq 1-4^{1-c/3}\,.
\]

\subsection{Lower bound for $X$}\label{ssec:lower_for_X}

Let $\mebfussy{X}$ be the random walk on $\ZZ$ defined as follows:
\begin{equation}\label{eqn:process_X'}
  \mebfussy X_0 = 0; \qquad \mebfussy X_t = \begin{cases}\mebfussy
    X_{t-1} + \mebfussy \xi_t & \text{with probability $1-p_n$} \\
    2\mebfussy X_{t-1} & \text{with probability $p_n$.}
  \end{cases}
\end{equation}
That is, $X = \mebfussy X \mod n$.

A natural approach to lower bound the mixing time for $X$ would be to
replicate the above argument for $Y$, using Chebychev's inequality
applied to $\mebfussy X$. However, the random number of jumps by time
$T_n^L-cw_n^L$ causes the variance of $\mebfussy X$ at this time to be
too large for this to work. Instead, we proceed by bounding the
expectation of $\abs{\mebfussy
  X_{T_n^L-cw_n^L}}$ 
and then using Markov's inequality to show that $X_{T_n^L-cw_n^L}$ has
negligible chance of belonging to a region of the state space situated
`opposite' $X_0=0$. We begin by proving the following lemma.

\begin{lem}\label{lem:expected_modulus}
  There exists a constant $a<\infty$ such that at time $T_n^L-cw_n^L$,
  \begin{equation}\label{eqn:expected_modulus}
    \Ex{\abs{\mebfussy X_{T_n^L-cw_n^L}}} \leq ane^{-c} \,.
  \end{equation}
\end{lem}

\begin{proof}
  Let $J_n(-c)$ be the number of jumps in $\mebfussy X$ by time
  $T_n^L-cw_n^L$. For large $n$, the distribution of $J_n(-c)$ is well
  approximated by a $\text{Poisson}(m_n(-c))$ distribution, with
  \begin{equation}\label{eqn:m_n}
    m_n(-c) = p_n(T_n^L-c w_n^L) = \ln (n/\sigma_{\mebfussy S}) - c \,,
  \end{equation}
  (using the definition of $T_n^L$ in Theorem~\ref{thm:main_result}
  and the equality $p_n = 1/(2n^\alpha)$). Note that under the
  assumptions of Theorem~\ref{thm:main_result}, $m_n(-c)\to\infty$ as
  $n\to\infty$. Conditional on the event $\{J_n(-c) = k\}$ we can
  express $\mebfussy X_{T_n^L-cw_n^L}$ as follows:
  \begin{equation}\label{eqn:X'_condtioned_on_jumps}
    \mebfussy X_{T_n^L-cw_n^L} \,|\, \{J_n(-c) = k\} =
    \sum_{i=1}^{k+1}2^{k+1-i}{\mebfussy S_i}^{(k)}\,,
  \end{equation}
  where ${\mebfussy S_i}^{(k)} \stackrel{d}{=} \mebfussy S_i\,|\, \{J_n(-c) = k\}$.  That
  is, ${\mebfussy S_i}^{(k)}$ is the additive increment in $\mebfussy X$ between jump times
  $\tau_{i-1}$ and $\tau_i$, with $\tau_0:=0 < \tau_1<\dots<\tau_k\leq
  \tau_{k+1}:= T_n^L-cw_n^L$.
  It is clear that for $i=1,\dots,k$ the random variables ${\mebfussy S_i}^{(k)}$
  have a common distribution, and that $\mathbb E[|{\mebfussy S}^{(k)}_{\hspace{-1mm}k+1}|]
  \leq \mathbb E[|{\mebfussy S_1}^{(k)}|]$ (since it is possible to have
  $\tau_k=\tau_{k+1}$). It follows from
  \eqref{eqn:X'_condtioned_on_jumps} that for $k\ge1$,
  \begin{equation}\label{eqn:next_X'_bound}
    \Ex{\left.\abs{\mebfussy X_{T_n^L-cw_n^L}} \,\right\vert\, J_n(-c) =
      k} \le 2^{k+1}\Ex{|{\mebfussy S_1}^{(k)}|}\,.
  \end{equation}
  We now deal with the cases of zero and non-zero $\mu$ separately.

  \smallskip $\bullet$ Case 1: $\mu\neq 0$.\\ Suppose that $k\ge1$.
  Let $b$ be an integer satisfying $B\subseteq[-2^b,2^b]$, where
  (recall that) $B$ is the support of $\mebfussy \xi$. Then
  \begin{align}
    \Ex{|{\mebfussy S_1}^{(k)}|} &\le 2^b \Ex{\tau_1-1\,|\, J_n(-c)
      = k} \nonumber \\
    &\le \frac{2^b(T_n^L-cw_n^L)}{k} \le
    \frac{2^{b+1}(T_n^L-cw_n^L)}{k+1}\,, \label{eqn:k+1_bound}
  \end{align}
  where the second inequality follows from the symmetry observation made
  just before \eqref{eqn:next_X'_bound}, and the last one uses the assumption that $k\ge
  1$. Combining \eqref{eqn:next_X'_bound} and \eqref{eqn:k+1_bound} we see
  that for $k\ge 1$,
  \begin{equation*}\label{eqn:Ex_X'_condtioned_on_jumps}
    \Ex{\left.\abs{\mebfussy X_{T_n^L-cw_n^L}} \,\right\vert\, J_n(-c) = k} \leq \frac{
      2^{b+k+2}(T_n^L-cw_n^L)}{k+1} \,.
  \end{equation*}
  Furthermore, note that this also (trivially) holds when $k=0$.

  Now average over the distribution of $J_n(-c)$:
  \begin{align*}
    \Ex{\abs{\mebfussy X_{T_n^L-cw_n^L}}} &\leq 2^{b+1} (T_n^L-cw_n^L)
    \sum_{k=0}^\infty
    \frac{e^{-m_n(-c)}m_n(-c)^k2^{k+1}}{k!(k+1)} \\
    &\leq \frac{2^{b+1}(T_n^L-cw_n^L)}{m_n(-c)}e^{m_n(-c)}
    = \frac{2^{b+1}}{p_n}\frac{ne^{-c}}{\sigma_{\mebfussy S}} \sim
    2^{b+1} ne^{-c} |\mu|
  \end{align*}
  for large $n$, thanks to the expression for $m_n(-c)$ in
  \eqref{eqn:m_n} and the relationship between $p_n$ and
  $\sigma_{\mebfussy S}$ in \eqref{eqn:var_S}. Taking $a=
  2^{b+1}|\mu|$ gives the required result.

  \medskip
  $\bullet$ Case 2: $\mu= 0$. \\
  In this case we know that $\Ex{{\mebfussy S_1}^{(k)}}=0$. Furthermore,
  \begin{align*}
    \Var{{\mebfussy S_1}^{(k)}}
    \le \Ex{\sigma^2 ( \tau_1 -1) \,|\, J_n(-c) = k } \le
    \frac{2\sigma^2T_n^L}{k+1} \,,
  \end{align*}
  using the same reasoning that led to
  \eqref{eqn:k+1_bound}. Chebychev's inequality then yields, for any
  positive $x$:
  \begin{align}
    \Ex{\abs{{\mebfussy S_1}^{(k)}}} \,\leq \, x + \int_x^\infty
    \prob{\abs{{\mebfussy S_1}^{(k)}}>s}
    \mathrm{d}s \, \leq \, x + \frac{2\sigma^2T_n^L}{(k+1)x}
    \,. \nonumber
  \end{align}
  Substituting $x=2\sigma\sqrt{T_n^L/(k+1)}$ we obtain
  \begin{equation*}\label{eqn:bound_when_mu_0}
    \Ex{\abs{{\mebfussy S_1}^{(k)}}} \leq 3\sigma\sqrt{\frac{T_n^L}{k+1}} \,.
  \end{equation*}
  Combining this with \eqref{eqn:X'_condtioned_on_jumps} we see that
  \begin{equation}\label{eqn:one_more_expectation_to_go}
    \Ex{\abs{\mebfussy X_{T_n^L-cw_n^L}}} \leq 6\sigma\sqrt{T_n^L}
    \Ex{\frac{2^{J_n(-c)}}{\sqrt{J_n(-c)+1}}}\,.
  \end{equation}
  Now consider the final expectation in
  \eqref{eqn:one_more_expectation_to_go}. Recalling that, for large
  $n$, $J_n(-c)$ is well-approximated by a Poisson distribution with
  mean $m_n(-c)$, we can bound this as follows:
  \begin{align}
    \Ex{\frac{2^{J_n(-c)}}{\sqrt{J_n(-c)+1}}} &\leq
    e^{m_n(-c)}\sum_{k=0}^{m_n(-c)}
    \frac{(2m_n(-c))^ke^{-2m_n(-c)}}{k!}\nonumber \\
    &\qquad + \frac{1}{\sqrt{m_n(-c)}} \sum_{k>{m_n(-c)}}
    \frac{(2m_n(-c))^ke^{-m_n(-c)}}{k!} \nonumber \\
    &\leq e^{m_n(-c)}\left[\prob{\Lambda< m_n(-c)} +
      \frac{1}{\sqrt{m_n(-c)}} \right] \label{eqn:Lambda}\,,
  \end{align}
  where $\Lambda\sim\text{Poisson}(2m_n(-c))$. A final application of
  Chebychev's inequality tells us that $\prob{\Lambda<{m_n(-c)}} \leq
  2/{m_n(-c)}$, and since $m_n(-c)\to \infty$ as $n\to \infty$ this
  term in \eqref{eqn:Lambda} is negligible.  Combining
  \eqref{eqn:one_more_expectation_to_go} and \eqref{eqn:Lambda} we
  therefore arrive at our desired result: for large $n$ we can write
  \begin{align*}
    \Ex{\abs{\mebfussy X_{T_n^L-cw_n^L}}} &\leq 12\sigma
    \sqrt{T_n^L}\frac{e^{m_n(-c)}}{\sqrt{m_n(-c)}} =
    \frac{12\sigma}{\sqrt{p_n}}\frac{ne^{-c}}{\sigma_{\mebfussy S}}
    \sim 12 ne^{-c}
  \end{align*}
  thanks once again to the expressions for $m_n(-c)$ and
  $\sigma_{\mebfussy S}$ in \eqref{eqn:m_n} and \eqref{eqn:var_S}.
\end{proof}

Recall that our aim in this section is to lower bound
\begin{equation}\label{eqn:need_set_A}
  \tv{\mathbb P(X_{T_n^L-cw_n^L}\in\cdot) - \pi_n(\cdot)}
  = \sup_A \bra{\pi_n(A) - \mathbb P(X_{T_n^L-cw_n^L}\in A)}\,.
\end{equation}
Define the set $D_n(c)$ to be those points in $\ZZ_n$ whose distance
from 0 (measured in the usual way between two numbers in $\ZZ_n$) is
greater than $e^{-c/2}n$. Note that
\[ \pi_n(D_n(c)) \ge 1- 2e^{-c/2}-1/n \,. \]
  
Using Markov's inequality and Lemma~\ref{lem:expected_modulus} we
obtain:
\begin{align*}
  \prob{X_{T_n^L-cw_n^L}\in D_n(c)} &\leq \prob{\abs{\mebfussy
      X_{T_n^L-cw_n^L}} > e^{-c/2}n} 
  \leq ae^{-c/2} \,.
\end{align*}
Thus
\[
\pi_n(D_n(c)) - \mathbb P(X_{T_n^L-cw_n^L}\in D_n(c)) \ge
1-(2+a)e^{-c/2}-1/n
\]
and so
\[ \liminf_{n\to\infty} \tv{\mathbb P(X_{T_n^L-cw_n^L}\in\cdot) -
  \pi_n(\cdot)} \geq 1-(2+a)e^{-c/2}\,.
\]


\section{Upper bounds}\label{sec:upper bound}

In this section we prove the second parts of
Theorems~\ref{thm:Y_cutoff} and \ref{thm:main_result}. Most of the
work here is required to prove the result for the subsampled chain
$Y$. Indeed, the result for $X$ follows quickly from this, as we now
demonstrate.

Let $J_n(c)$ denote the number of jumps in $X$ by time $T_n^R+cw_n^R$:
for large $n$ this is once again well approximated by a Poisson random
variable, this time with mean $m_n(c) = T_n + c\sqrt{T_n}$. Assuming
Theorem~\ref{thm:Y_cutoff} to be true, we know that the subsampled
chain $Y$ is well mixed after $T_n+c$ jumps; we proceed by considering
whether or not this number of jumps has occurred by our time of
interest, $T_n^R+cw_n^R$. For ease of display, in the next few lines
we shall write $\tau = \tau_{T_n+c}$ for the
(random) time at which $X$ jumps for the $(T_n+c)$-th time.
\begin{align}
  \tv{\mathbb P(X_{T_n^R+cw_n^R}\in\cdot) - \pi_n(\cdot)} &\le
  \Ex{\tv{\mathbb P(X_{T_n^R+cw_n^R}\in\cdot) -
      \pi_n(\cdot)}\,;\, \tau\leq {T_n^R+cw_n^R}} \nonumber \\
&\qquad + \prob{\tau>{T_n^R+cw_n^R}} \label{eqn:tv_split}
\end{align}
due to total variation being bounded above by 1. We now use the fact
that total variation is non-increasing over time to bound the
expectation term:
\begin{align}
  \Ex{\tv{\mathbb P(X_{T_n^R+cw_n^R}\in\cdot) - \pi_n(\cdot)}\,;\, \tau\leq {T_n^R+cw_n^R}} &=
  \Ex{\sum_{k=1}^{T_n^R+cw_n^R}\tv{\mathbb P(X_{T_n^R+cw_n^R}\in\cdot) -
      \pi_n(\cdot)}\indev{\tau=k}} \nonumber \\
  &\le \Ex{\sum_{k=1}^{T_n^R+cw_n^R}\tv{\mathbb P(X_k\in\cdot) -
      \pi_n(\cdot)}\indev{\tau=k}} \nonumber \\
  &= \Ex{\tv{\mathbb P(X_\tau\in\cdot) -
      \pi_n(\cdot)}\,;\,\tau\le {T_n^R+cw_n^R}} \nonumber \\
  &\le \tv{\prob{X_\tau\in\cdot}-
    \pi_n(\cdot)} \nonumber \\
  &= \tv{\prob{Y_{T_n+c}\in\cdot}- \pi_n(\cdot)} \,.\label{eqn:tv_split2}
\end{align}

The first term in \eqref{eqn:tv_split} is thus controlled by part 2 of
Theorem~\ref{thm:Y_cutoff}. Furthermore, using Chebychev's inequality
once again:
\begin{align*}
  \prob{\tau>T_n^R+cw_n^R} &= \prob{J_n(c)< T_n+c} \leq
  \prob{\abs{J_n(c) - m_n(c)}
    \geq c\bra{\sqrt{T_n}-1} } \\
  &\leq \frac{m_n(c)}{c^2 T_n+O(\sqrt{T_n})} \to \frac{1}{c^2}
  \quad\text{as $n\to\infty$.}
\end{align*}
Putting all of the above together we complete the proof of
Theorem~\ref{thm:main_result}: for sufficiently large $c>0$,
\[ \limsup_{n\to\infty} \tv{\mathbb P(X_{T_n^R+cw_n^R}\in\cdot) -
  \pi_n(\cdot)} \leq \frac{2}{c^2} \,. \]

\subsection{Upper bounds and representation theory}

Our basic method for obtaining upper bounds on the mixing times of our
processes is to employ the techniques developed by
\citet{Diaconis.Shahshahani1981} for analysing random walks on groups.
Given a probability $Q$ on a finite group $G$, and a representation
$\rho$ of $G$, we can form the Fourier transform $\hat Q(\rho)$ of $Q$
at $\rho$ by setting
$$
\hat Q(\rho):= \sum_{g\in G} Q(g)\rho(g)\,.
$$
The following Upper Bound Lemma \citep{Diaconis1988} then allows one
to compute an explicit upper bound for the total variation distance
between a probability $Q$ on $G$ and the uniform distribution $\pi$.

\begin{lem}\label{lem:UBL}
  Given a probability $Q$ on a finite group $G$, we have
$$
\tv{Q-\pi}^2 \leq \frac{1}{4}\sum \degree(\rho) \tr(\hat Q(\rho) \hat
Q(\rho)^*),
$$
where $A^* = (\overline{a_{ji}})$ denotes the complex conjugate
transpose of the matrix $A = (a_{ij})$, $\tr$ denotes the trace
function on square matrices, and the sum is taken over all non-trivial
irreducible representations $\rho$ of $G$.
\end{lem}


Since the Fourier transform behaves well with respect to convolution,
this lemma provides a practical tool for bounding the mixing time of a
random walk on a group.
Although $Y$ is not strictly a random walk on the additive group
$(\ZZ_n,+)$, the measure giving the distribution of $Y_k$ \emph{can}
be expressed as the convolution of measures. Here the representation
theory is particularly straightforward:
the Upper Bound Lemma becomes
\begin{equation}\label{eqn:UBL}
  \tv{Q-\pi}^2 \leq \frac{1}{4}\sum_{s=1}^{n-1} |\hat 
  Q(\rho_s)|^2\,,
\end{equation}
where the representations $\rho_0,\rho_1,\ldots,\rho_{n-1}$ all have
degree $1$, and are completely determined by the following equations:
$$
\rho_s(1) := e^{i\frac{2\pi}{n}s} \textrm{ for } 0\leq s\leq n-1\,.
$$


Recall from \eqref{eqn:Y_k_and_tilde_S} that (with $Y_0=0$), $Y_k =
\sum_{j=1}^k2^j\mebfussy S_{k+1-j} \mod n$. The measure $P_k$ giving
the distribution of $Y_k$ is the convolution of the measures
$\lambda_j$ given by $\lambda_j(2^ja\mod n) = \prob{S_1=a}$ for every
$j,a$, so we begin by calculating the Fourier transforms of the
$\lambda_j$.  To ease notation, for each $1\leq j \leq k$ and $0\leq s
\leq n-1$, set
$$
\omega_{s,j} = \rho_s(2^j) = e^{i\frac{2\pi}{n}2^js}
$$
and note that for any $j,s$ we have $\omega_{s,j}^n = 1$.  Then for
each $0\leq s \leq n-1$,
\begin{align*}
  \hat{\lambda}_j(\rho_s) &= \sum_{a=0}^{n-1}
  \omega_{s,j}^a\prob{S_1=a}= \sum_{a=0}^{n-1} \omega_{s,j}^a \sum_{d
    \in
    \ZZ}\prob{\mebfussy S_1=a + dn} \\
  &= \sum_{d \in \ZZ}\sum_{a=0}^{n-1}
  \omega_{s,j}^{a+dn}\prob{\mebfussy S_1=a + dn} = \sum_{a\in\ZZ}
  \omega_{s,j}^a\prob{\mebfussy S_1=a} = G_{\mebfussy
    S}(\omega_{s,j})\,,
\end{align*}
where $G_{\mebfussy S}$ is the probability generating function (PGF)
of $\mebfussy S$. It follows from its definition in
\eqref{eqn:Y_k_and_tilde_S} as a random sum of random step lengths
that this satisfies
\begin{equation}\label{eqn:pgf_tilde_S}
  G_{\mebfussy S}(\omega_{s,j}) = \frac{p_n}{1-(1-p_n)G_{\mebfussy\xi}(\omega_{s,j})} \,,
\end{equation}
where $G_{\mebfussy\xi}$ is the PGF of $\mebfussy\xi$.

When we substitute into the Upper Bound Lemma \ref{lem:UBL}, we are
interested in the modulus squared of such expressions, by Equation
\eqref{eqn:UBL}. The modulus of the top line squared is $p_n^2$, and
the modulus of the bottom line squared is
\begin{align*}
  (1-(1-p_n)&G_{\mebfussy\xi}(\omega_{s,j}))(1-(1-p_n)\overline{G_{\mebfussy\xi}(\omega_{s,j})}) \\
  &=
  1-(1-p_n)\bra{G_{\mebfussy\xi}(\omega_{s,j})+\overline{G_{\mebfussy\xi}(\omega_{s,j})}}
  +
  (1-p_n)^2G_{\mebfussy\xi}(\omega_{s,j})\overline{G_{\mebfussy\xi}(\omega_{s,j})} \\
  &= 1-2(1-p_n) \Re\bra{G_{\mebfussy\xi}(\omega_{s,j})} + (1-p_n)^2
  \abs{G_{\mebfussy\xi}(\omega_{s,j})}^2 \,.
\end{align*}
Combining all of the above leads to the following upper bound for the
total variation distance at time $k$:
\begin{equation}\label{eqn:UB_before_split}
  \norm{\prob{Y_{k}\in\cdot} - \pi_n(\cdot)}^2\leq 
  \frac{1}{4}\sum_{s=1}^{n-1}\prod_{j=1}^{k} \frac{p_n^2}{1-2(1-p_n)
    \Re\bra{G_{\mebfussy\xi}(\omega_{s,j})} + (1-p_n)^2
    \abs{G_{\mebfussy\xi}(\omega_{s,j})}^2}\,.
\end{equation}


\subsection{Strategy for analysing the upper bound}

In order to establish a cutoff for $Y$, we need to control the right
hand side of \eqref{eqn:UB_before_split} around time $T_n =
\log_2(n/\sigma_{\mebfussy S})$. To that end, we define for $c\in\NN$
a function $U_n(c)$ by
\begin{equation}\label{eqn:U_defn}
  U_n(c) =  \sum_{s=1}^{n-1} \prod_{j=1}^{T_n+c} \phi_n(s,j)
\end{equation}
where
\begin{equation}\label{eqn:phi_defn}
  \phi_n(s,j):= \frac{p_n^2}{1-2(1-p_n)
    \Re\bra{G_{\mebfussy\xi}(\omega_{s,j})} + (1-p_n)^2
    \abs{G_{\mebfussy\xi}(\omega_{s,j})}^2} \in(0,1]\,,
\end{equation}
and note that, thanks to \eqref{eqn:UB_before_split},
Theorem~\ref{thm:Y_cutoff} will be proved if we can show that (for odd
$n$) $\limsup_{n\to\infty} U_n(c)\leq O(4^{-c})$.

Our strategy for bounding $U_n(c)$ involves identifying for each
$1\leq s \leq n-1$ enough values $j$ for which $\phi_n(s,j)$ is
sufficiently small to provide a useful upper bound.  In order to do
this, it is convenient to first reparametrise, so we let $Z_n$ be a
random variable uniformly distributed on the set
$\{s/n:s=1,\dots,n-1\}\subset[0,1]$.  Then we may write
\begin{equation}\label{eqn:U_as_expectation}
  U_n(c) = \Ex{f_n(Z_n,T_n+c)}, \quad\text{where } f_n(x,t) := (n-1) 
  \prod_{j=1}^{t} \phi_n(nx,j)\,.
\end{equation}
The second step is to split the analysis of the function $f_n$ into
two stages by splitting the range of $x$ into two pieces.  In order to
do this, let $L$ be an integer satisfying $2\alpha L>1$, and once
again let $b$ be an integer satisfying $B\subseteq[-2^b,2^b]$.  We
define a finite lattice $\mathcal{L}$ of points in $[0,1]$ by
\[ \mathcal{L} = \set{\frac{k}{2^{L+b}}\;:\;k=0,\dots,2^{L+b}}\,. \]
Now choose some $\eps\in(0,1/(2^{L+b}))$, and define the set $\mathcal
L_\eps$ to be the intersection of $[0,1]$ with
\[ \bigcup_{x\in\mathcal L} \squ{x-\frac{\eps}{2},x+\frac{\eps}{2}}
\,. \] Importantly, $\mathcal L_\eps$ depends only on $\alpha$, $B$
and $\eps$, but not on $n$. We now proceed to bound $f_n(x,T_n+c)$ by
considering in turn the cases where $x$ does and does not belong to
the set $\mathcal L_\eps$.

\subsection{Controlling $f_n$ for $x\notin\mathcal L_\eps$}

For $x\notin\mathcal L_\eps$ we see that $2\pi 2^ja x\neq 0 \mod 2\pi$
for any $j=1,2,\dots,L$ and $a\in B$.  Thus $\cos(2\pi2^jax)$ is
bounded away from 1 for all such $x$ and $j$, and we can write
\[ \Re\bra{G_{\mebfussy\xi}(e^{i2\pi2^j x})} = \sum_{a=0}^{2^b}
\prob{\abs{\mebfussy\xi}=a}\cos\bra{2\pi2^jax}\leq 1-\kappa(x)\,, \]
for all $j=1,\dots,L$, where $\kappa(x)$ is strictly positive.

Substituting this into the expression for $\phi_n$ in
\eqref{eqn:phi_defn}, and lower-bounding the modulus squared of a
complex number by the square of its real part, we obtain:
\begin{align*} \phi_n(nx,j)& \leq \frac{p_n^2}{1-2(1-p_n)
    \Re\bra{G_{\mebfussy\xi}(e^{i2\pi2^j x})} + (1-p_n)^2 \abs{G_{\mebfussy\xi}(e^{i2\pi2^j x})}^2} \\
  &\leq \frac{p_n^2}{1-2(1-p_n) \Re\bra{G_{\mebfussy\xi}(e^{i2\pi2^j
        x})} + (1-p_n)^2
    \Re\bra{G_{\mebfussy\xi}(e^{i2\pi2^j x})}^2} \\
  &= \left( \frac{p_n}{1-(1-p_n) \Re\bra{G_{\mebfussy\xi}(e^{i2\pi2^j
          x})}}\right)^2 \leq \left( \frac{p_n}{1-(1-p_n)
      (1-\kappa(x))}\right)^2 =O(p_n^2) \,.
\end{align*}

Since $\phi_n(nx,j)\in(0,1]$, it follows that for $x\notin\mathcal
L_\eps$ and for $n$ sufficiently large that $T_n+c\ge L$,
\[ f_n(x,T_n+c) = (n-1) \prod_{j=1}^{T_n+c} \phi_n(nx,j) \leq (n-1)
\prod_{j=1}^L \phi_n(nx,j) \leq O(n^{1-2\alpha L}) \,. \] Thanks to
our choice of $L>1/2\alpha$ we can now use Fatou's Lemma to deduce
that \begin{equation}\label{eqn:U_to_zero_off_lattice}
  \limsup_{n\to\infty} \Ex{f_n(Z_n,T_n+c);Z_n\notin\mathcal L_\eps} =
  0 \,.
\end{equation}

\subsection{Controlling $f_n$ for $x\in\mathcal L_\eps$}

It remains to deal with $\Ex{f_n(Z_n,T_n+c);Z_n\in\mathcal L_\eps}$.
We begin by writing (for any $t\in\NN$)
\begin{align}
  \Ex{f_n(Z_n,t);Z_n\in\mathcal L_\eps} &= \frac{1}{n-1}
  \sum_{k=1}^{2^{L+b}-1} \sum_{r \geq
    1}f_n\bra{\frac{r}{n},\,t}\indev{\abs{\tfrac{k}{2^{L+b}}-\tfrac{r}{n}}\leq
    \tfrac{\eps}{2}} \nonumber \\
  & \qquad + \frac{1}{n-1}\sum_{r \geq
    1}\bra{f_n\bra{\frac{r}{n},\,t}+f_n\bra{1-\frac{r}{n},\,t}}\indev{\tfrac{r}{n}\leq
    \tfrac{\eps}{2}}\,, \label{eqn:first_expansion}
\end{align}
where the last sum deals with the two end intervals in $\mathcal
L_\eps$. Consider first of all the double sum here. Since $n$ is odd,
the shortest possible distance between any point $r/n$ and the lattice
point $k/2^{L+b}$ is $1/(n2^{L+b})$; as $f_n$ is non-negative we can
therefore upper bound the inner sum by summing over a lattice of size
$1/(n2^{L+b})$ centred around $k/2^{L+b}$ as follows:
\begin{equation}\label{eqn:finer_lattice}
  \sum_{r \geq
    1}f_n\bra{\frac{r}{n},\,t}\indev{\abs{\tfrac{k}{2^{L+b}}-\tfrac{r}{n}}\leq
    \tfrac{\eps}{2}} \leq \sum_{\substack{r =-\infty \\r\neq 0}}^\infty
  f_n\bra{\frac{k}{2^{L+b}}-\frac{r}{n2^{L+b}},\,t}\,.
\end{equation}

Thus \eqref{eqn:first_expansion} can be bounded as follows:
\begin{align}
  \Ex{f_n(Z_n,t);Z_n\in\mathcal L_\eps}
  &\leq \frac{1}{n-1} \sum_{k=1}^{2^{L+b}-1} \sum_{\substack{r
      =-\infty \\r\neq 0}}^\infty
  f_n\bra{\frac{k}{2^{L+b}}-\frac{r}{n2^{L+b}},\,t} +\frac{2}{n-1}
  \sum_{r=1}^\infty f_n\bra{\frac{r}{n},\,t} \,, \label{eqn:two_sums}
\end{align}
where we have used the symmetry of the functions $f_n$ at either end
of the interval $[0,1]$ to rewrite the expression for the end
intervals.  Now replace $t$ by $T_n+c$, and consider the function
$f_n$ in the double sum above:
\begin{align}
  f_n\bra{\frac{k}{2^{L+b}}-\frac{r}{n2^{L+b}},\,T_n+c} &= (n-1)
  \prod_{j=1}^{T_n+c} \phi_n\bra{\frac{nk-r}{2^{L+b}},\,j}\leq (n-1)
  \phi_n\bra{\frac{nk-r}{2^{L+b}},\,T_n+c} \,. \label{eqn:prod}
\end{align}
Here we have bounded the product by a single term, once again making
use of the fact that $\phi_n$ takes values in $(0,1]$. Since
$\phi_n(s,j)$ involves $s$ and $j$ only through the function
$G_{\mebfussy\xi}(\omega_{s,j})$, where $\omega_{s,j} = \exp(2\pi i
2^j s/n)$, we have (for sufficiently large $n$) that the bound in
\eqref{eqn:prod} is a function of
\begin{align*}
  \exp\bra{2\pi i \frac{2^{T_n+c}}{n}\bra{\frac{nk-r}{2^{L+b}}}} =
  \exp\bra{-2\pi i \frac{2^{T_n+c}r}{n2^{L+b}}} =
  \exp\bra{-\frac{2^{1+c-(L+b)}\pi ir}{\sigma_{\mebfussy S}}}\,.
\end{align*}
The second equality simply uses the definition of $T_n$, while the
first results from shifting the argument of the exponential function
by $2\pi i k2^{T_n+c-(L+b)}$. (For large enough $n$ this is an integer
multiple of $2\pi i$, thanks to the finiteness of $L$ and $b$ and the
assumption that $T_n\to\infty$.)

Writing
\[ \theta_{nrc} = \frac{2^{1+c-(L+b)}\pi r}{\sigma_{\mebfussy
    S}}\,, \] (where recall that $\sigma_{\mebfussy S}$ depends on
$n$) we therefore need to upper bound the function
\[ \phi_n\bra{\frac{nk-r}{2^{L+b}},\,T_n+c} = \frac{p_n^2}{1-2(1-p_n)
  \Re\bra{G_{\mebfussy\xi}(e^{-i\theta_{nrc} })} + (1-p_n)^2
  \abs{G_{\mebfussy\xi}(e^{-i\theta_{nrc}})}^2}\,. \] Now note that
\[ G_{\mebfussy\xi}(e^{-i\theta_{nrc}}) =
\Ex{e^{-i\mebfussy\xi\theta_{nrc}}} \quad\text{ and thus }\quad
\Re\bra{G_{\mebfussy\xi}(e^{-i\theta_{nrc}})} =
\Ex{\cos\bra{\mebfussy\xi\theta_{nrc}}} \,. \] Similarly,
\[
\abs{G_{\mebfussy\xi}(e^{-i\theta_{nrc}})}^2 =
\Ex{\cos\bra{\mebfussy\xi\theta_{nrc}}}^2 +
\Ex{\sin\bra{\mebfussy\xi\theta_{nrc}}}^2 \,.
\]

Since $p_n\to 0$ as $n\to\infty$, we see from \eqref{eqn:var_S} that
$\sigma_{\mebfussy S}\to\infty$ and so
$\theta_{nrc}\to 0$.  Using the Taylor expansions of cosine and sine
the above can be approximated by
\begin{align*}
  \Ex{\cos\bra{\mebfussy\xi\theta_{nrc}}} &=
  1-\frac{(\mu^2+\sigma^2)\theta_{nrc}^2}{2} + O(\theta_{nrc}^4) \,;
  \quad \Ex{\sin\bra{\mebfussy\xi\theta_{nrc}}} = \mu\theta_{nrc} +
  O(\theta_{nrc}^3) \,.
\end{align*}
Neglecting terms of $O(\theta_{nrc}^3)$ we arrive at
\begin{align*}
  \phi_n\bra{\frac{nk-r}{2^{L+b}},\,T_n+c} &\sim
  \frac{p_n^2}{1-(1-p_n)\squ{2-(\mu^2+\sigma^2)\theta_{nrc}^2} +
    (1-p_n)^2\squ{1-(\mu^2+\sigma^2)\theta_{nrc}^2 +
      \mu^2\theta_{nrc}^2}}
  \\
  &= \frac{p_n^2}{p_n^2+(1-p_n)(\mu^2+\sigma^2p_n)\theta_{nrc}^2} =
  \frac{1}{1+ 4^{1+c-(L+b)}\pi^2r^2} \,.
\end{align*}

We now combine this bound with that in \eqref{eqn:prod} and insert
into \eqref{eqn:two_sums} (using an identical argument for the second
sum there):
\begin{align}
  \limsup_{n\to\infty}\Ex{f_n(Z_n,t);Z_n\in\mathcal L_\eps}& \leq
  \sum_{k=1}^{2^{L+b}-1} \sum_{\substack{r=-\infty\\r\neq0}}^\infty
  \frac{1}{1+ 4^{1+c-(L+b)}\pi^2r^2} + 2 \sum_{r=1}^\infty \frac{1}{1+
    4^{1+c-(L+b)}\pi^2r^2} \nonumber \\
  &= 2^{L+b}
  \bra{2^{L+b-(1+c)}\coth(2^{L+b-(1+c)})-1} \nonumber \\
  & \sim 2^{L+b}\,\frac{4^{L+b-(1+c)}}{3} \quad\text{as
    $c\to\infty$,} \label{eqn:final_bound}
\end{align}
where we have made use of the well-known identity
\citep[p. 334]{Apostol1974}
\[
\coth(x)=\frac{1}{x}+2x\sum_{r=1}^\infty\frac{1}{x^2+\pi^2r^2} \qquad
(x>0)\,.
\]
Combining \eqref{eqn:U_to_zero_off_lattice} and
\eqref{eqn:final_bound} yields the required result
\[ \limsup_{n\to\infty} U_n(c) = \limsup_{n\to\infty}
\Ex{f_n(Z_n,T_n+c);Z_n\in\mathcal L_\eps} \leq O(4^{-c}) \quad\text{as
  $c\to\infty$,} \] and thanks to the comment after
\eqref{eqn:phi_defn}, this completes the proof of part 2 of
Theorem~\ref{thm:Y_cutoff}.


\section{Concluding remarks}\label{sec:conclusions}

We have shown that the subsampled process $Y$ exhibits a
cutoff when the probability $p_n$ of jumping takes the form
$p_n=1/(2n^\alpha)$, for a range of $\alpha$ which depends upon the
mean of our step distribution ($\alpha\in(0,2)$ when $\mu=0$, and
$\alpha\in(0,1)$ otherwise). Furthermore, our original chain of
interest $X$ exhibits a pre-cutoff, with mixing time $\tmix{(n)}(\eps)
= \Omega(n^\alpha \ln n)$.

We have not yet said much about the mixing time of either process when
$\alpha$ takes values on the boundary of these intervals, however. If
$\alpha=0$ then part 1 of Theorems~\ref{thm:Y_cutoff} and
\ref{thm:main_result} (which do not rely on $p_n$ tending to zero)
still hold; however, our argument for upper bounding the mixing time
of $Y$ (and hence of $X$) breaks down, since a sufficiently fine
lattice $\mathcal L$ does not exist. (An upper bound of $O(\ln n
\ln\ln n)$ can be obtained for the mixing time of $X$ by employing the
method of \cite{Chung1987}.) On the other hand, if $\alpha$ takes the
value at the upper boundary of the relevant interval then
$n/\sigma_{\mebfussy S} = O(1)$, and thus $T_n$ is asymptotically
independent of $n$: in this case it is relatively easy to show that
$Y$ mixes in constant time (and so no longer exhibits a cutoff), and
that $X$ has mixing time of $\Omega(n^\alpha)$.

It is of course possible to generalise the process considered in this
paper in a number of ways. For example, changing the form of $p_n$ to
$1/(\beta n^\alpha)$ for some constant $\beta >1$ has no effect on the
cutoff result for $Y$. Similarly, changing the transitions of $X$ so
that jumps involve multiplying by some (fixed) $k\geq 2$ (and
considering only those $n$ for which the resulting process still has a
uniform equilibrium distribution) presumably has the effect of
changing the base of the logarithm in the cutoff time $T_n$ for $Y$
from 2 to $k$; Theorem~\ref{thm:main_result} should also still hold,
with the factor of $\ln2$ in $T_n^L$ being replaced by $(\ln
k)/(k-1)$. More interesting would be an analysis of a process $X$ for
which the multiplication factor is not deterministic, and for which
the resulting subsampled chain $Y$ does not have a distribution given
by simple convolution; for example where jumps in $X$ correspond to
multiplication by $a_t$ (again with probability $p_n\to 0$), with
$a_t$ being uniformly chosen from the set $\set{2,(n+1)/2}$.

\section*{Acknowledgements}

Some of the ideas in this work arose during investigations into random
walks on $\ZZ_n$ by Sam Wright, who was supported by Nuffield Science
Undergraduate Research Bursary URB/40605. The authors would also like
to express their gratitude to John Payne, whose numerical calculations
for a particular instance of our process provided some useful early
insights into the behaviour of the upper bound for $Y$ in
Section~\ref{sec:upper bound}.


\bibliographystyle{chicago} \bibliography{rings}

\vspace{1cm}
\noindent
Department of Mathematics\\ University of York\\ York, YO10 5DD\\ UK

\end{document}